\documentclass[11pt, reqno]{amsart}
\usepackage{amsmath, amsthm, amscd, amsfonts, amssymb, graphicx, color, mathrsfs}
\usepackage[bookmarksnumbered, colorlinks, plainpages]{hyperref}
\usepackage[all]{xy}
\usepackage{slashed}

\newtheorem{theorem}{Theorem}[section]

\theoremstyle{definition}
\newtheorem{definition}[theorem]{Definition}
\newtheorem{example}[theorem]{Example}

\theoremstyle{remark}
\newtheorem{remark}[theorem]{Remark}
\numberwithin{equation}{section}

\begin{document}
\setcounter{page}{1}

\title[ Besov continuity for pseudo-differential operators]{ Besov continuity for pseudo-differential operators on compact homogeneous manifolds}

\author[D. Cardona]{Duv\'an Cardona}
\address{
  Duv\'an Cardona:
  \endgraf
  Department of Mathematics  
  \endgraf
  Pontificia Universidad Javeriana
  \endgraf
  Bogot\'a
  \endgraf
  Colombia
  \endgraf
  {\it E-mail address} {\rm cardonaduvan@javeriana.edu.co;
duvanc306@gmail.com}
  }


\dedicatory{}



\date{Received: xxxxxx;  Revised: yyyyyy; Accepted: zzzzzz.
\newline \indent }

\begin{abstract}
In this paper we study the Besov continuity  of pseudo-differential operators on compact homogeneous manifolds $M=G/K.$ We use the global quantization of these operators in terms of the representation theory of  compact homogeneous manifolds.
\end{abstract} \maketitle

\tableofcontents

\section{Introduction}

In this work we study the mapping properties of pseudo-differential operators on Besov spaces defined on compact Lie groups. The Besov spaces $B^r_{p,q}$ arose from attempts to unify the various definitions of
several fractional-order Sobolev spaces. In order to illustrate the mathematical relevance of the Besov spaces, we recall that 
from the context of the applications, a function belonging to some of these spaces admits a  decomposition of the form
\begin{equation}
f=f_0+\sum_{j\geq 1}g_j, \,\,\,g_j=f_{j+1}-f_{j},
\end{equation}
satisfying
\begin{equation}
\Vert f\Vert_{B^r_{p,q}}:=\Vert \{2^{jr} \Vert g_j \Vert_{L^p}\}_{j=0}^{\infty} \Vert_{l^q(\mathbb{N})}<\infty,\,\,\,g_0=f_0,\,0<p,q\leq \infty,\,r\in\mathbb{R},
\end{equation}
where  the sequence $(f_j)_{j}$ consists of approximations to the data (or the unknown) $f$ of a given problem at various levels of resolution indexed by $j.$    In practice such approximations can be defined by using the Fourier transform and this description is useful in the numerical analysis of wavelet methods, and some areas of applied mathematics as signal analysis and image processing. In the field of numerical analysis, multi-scale and wavelet decompositions, Besov spaces have been used for three main task: preconditioning large systems arising from the discretization of elliptic differential problems, adaptive approximations of functions which are not smooth, and sparse representation of initially full matrices arising in the discretization of integral equations (see A. Cohen \cite{Cohen}).  \\
\\
In our work, we are interesting in the study of pseudo-differential problems associated to Besov spaces on compact Lie groups and more generally compact homogeneous manifolds. More precisely, we want to study the action of pseudo-differential operators on these spaces. We will use the formulation of  Besov spaces on compact homogeneous manifolds in terms of representation theory as in (\cite{RuzBesov}).  Our main goal is to show that, under certain conditions, the $L^{p}$ boundedness of Fourier multipliers on compact homogeneous manifolds gives to rise to results of continuity for pseudo-differential operators on Besov spaces. In our analysis we use the theory of global pseudo-differential operators on compact Lie groups and on compact homogeneous manifolds, which was initiated in the PhD thesis of  V. Turunen and  was extensively developed by M. Ruzhansky and V. Turunen in \cite{Ruz}. In this theory, every operator $A$ mapping $C^{\infty}(G)$ itself, where $G$ is a compact Lie group, can be described in terms of  representations of $G$ as follows. Let $\widehat{G}$ be the unitary dual of $G$ (i.e, the set of equivalence classes of continuous irreducible  unitary representations on $G$), the Ruzhansky-Turunen approach establish that  $A$ has associated a {\em matrix-valued global (or full) symbol} $\sigma_{A}(x,\xi)\in \mathbb{C}^{d_\xi \times d_\xi},$  $[\xi]\in \widehat{G},$ on the non-commutative phase space $G\times\widehat{G}$ satisfying
\begin{equation}
\sigma_A(x,\xi)=\xi(x)^{*}(A\xi)(x):=\xi(x)^{*}[A\xi_{ij}(x)]_{i,j=1}^{d_\xi}.
\end{equation}
Then it can be shown that the operator $A$ can be expressed in terms of such a symbol as \cite{Ruz}
\begin{equation}\label{mul}Af(x)=\sum_{[\xi]\in \widehat{G}}d_{\xi}\text{Tr}[\xi(x)\sigma_A(x,\xi)\widehat{f}(\xi)]. 
\end{equation}  
In the last five years, applications of this theory have been considered by many authors.  Advances in this framework includes the characterization of H\"ormander classes $S^{m}_{\rho,\delta}(G)$ on compact Lie groups in terms of the representation theory of such groups (c.f \cite{RWT} ), the sharp G\"arding inequality on compact Lie groups, (c.f \cite{RT0}), the behavior of Fourier multipliers in $L^{p}(G)$ spaces (c.f. \cite{Ruz3}), global functional calculus of operator on Lie groups (c.f \cite{RW2}), $r$-nuclearity of operators,  Grothendieck-Lidskii formula and nuclear traces of operators on compact Lie groups (c.f. \cite{DR,DR1,DR3}), the Gohberg lemma, characterization of compact operators, and the essential spectrum of operators on $L^2$ (c.f \cite{DAR}), $L^{p}$-boundedness of pseudo-differential operators in H\"ormander classes (c.f. \cite{DR4}), Besov continuity and nuclearity of Fourier multipliers on compact Lie groups (c.f \cite{Cardona,cardona2,Cardona3}), diffusive wavelets on groups and homogeneous spaces \cite{EW}, and recently, a reformulation of Ruzhansky and Turunen approach on the pseudo-differential calculus in compact Lie groups (c.f. Fischer, V \cite{Fischer}), including a version of the Calder\'on-Vaillancourt Theorem in this framework. \\

In the euclidean case of $\mathbb{R}^{n},$ the H\"ormander's symbol class $S^{m}_{\rho,\delta}(\mathbb{R}^n),$ $m\in\mathbb{R}$ and $0\leq \delta, \rho\leq 1,$ is defined by those functions $a(x,\xi),$ $x,\xi\in\mathbb{R}^n$ satisfying
\begin{equation}
|\partial_{x}^{\beta}\partial_{\xi}^{\alpha}a(x,\xi)|\leq C_{\alpha,\beta}\langle \xi\rangle^{m-\rho|\alpha|+\delta|\beta|}, \,\,\,\alpha,\beta\in\mathbb{N}^{n},
\end{equation}
$\langle\xi \rangle=(1+|\xi|)^{2}.$ The corresponding pseudo-differential operator $A$ with symbol $a(\cdot,\cdot)$ is defined on the Schwartz space $\mathscr{S}(\mathbb{R}^{n})$ by
\begin{equation}\label{pseudorn}
Af(x)=\int e^{i2\pi x\cdot \xi}a(x,\xi)\widehat{f}(\xi)d\xi.
\end{equation}
Consequently, on every differential manifold $M,$ pseudo-differential operators $A:C^{\infty}(M)\rightarrow C^{\infty}(M)$ associated to H\"ormander classes $S^{m}_{\rho,\delta}(M),$ $0\leq \delta<\rho \leq 1,$  $\rho\geq 1-\delta$ can be defined by the use of coordinate charts. When $M=G$ is a compact Lie group and $1-\rho\leq \delta,$ the exceptional results in \cite{RWT} gives an equivalence of the H\"ormander classes defined by charts and H\"ormander classes defined in terms of the representation theory of the group $G$.  \\

If $K$ is a closed subgroup of a compact Lie group $G,$ there is a canonical way to identify the quotient space $M=G/K$ with a analytic manifold.  Besov spaces on compact Lie groups and general compact homogeneous manifolds where introduced in terms of representations and analyzed  in \cite{RuzBesov}, they form scales $B^r_{p,q}(M)$ carrying three indices $r\in\mathbb{R},$ $0<p,q\leq \infty.$ For $1\leq p<\infty,$ $1\leq q\leq \infty$ the Besov spaces $B^{r}_{p,q}(M)$ coincide with the Besov spaces defined trough of localization with the euclidean space $B^{r}_{p,q}(\mathbb{R}^n).$ It is well known that if $a\in S^{m}_{1,\delta}(\mathbb{R}^n),$ $0\leq \delta <1,$ then the corresponding operator $A:B^{r+m}_{p,q}(\mathbb{R}^n)\rightarrow B^{r}_{p,q}(\mathbb{R}^n)$ is bounded for $1<p<\infty,$ $1\leq q<\infty$ and $r\in\mathbb{R},$ (c.f Bordaud \cite{Bordaud}, and Gibbons \cite{Gibbons}).  This implies the same result for compact Lie groups when $1< p<\infty$ and $1\leq q\leq \infty.$ With this fact in mind, in order to obtain Besov continuity for operators, we concentrate our attention to pseudo-differential operators $A$ whose symbols $a=\sigma_{A}$ have limited regularity almost in one of the variables $x,\xi.$ (Since, $\xi$ in the case of compact Lie groups has a discrete nature, the notion of differentiation is related with difference operators). The results of this paper have been announced in \cite{cardona2}. The Besov continuity of multipliers in the context of graded Lie groups has been considered by the author and M. Ruzhansky in \cite{CardonaRuzhansky}.\\

This paper is organized as follows. In section \ref{mainresults} we present and briefly discuss our main theorems. In Section \ref{Preliminaries}, we summarizes basic properties on the harmonic analysis in compact Lie groups including the Ruzhansky-Turunen theory of global pseudo-differential operators on compact-Lie groups and the definition of Besov spaces on such groups. Finally, in section \ref{proof} we proof our results on the boundedness of invariant and non-invariant pseudo-differential operators on Besov spaces and some examples on differential problems are given in Subsection \ref{examples}.
\section{Main results}\label{mainresults}
\noindent In this section we present and briefly discuss our main theorems. The following is a generalization of Theorem 1.2 of \cite{Cardona} to the case of homogeneous compact manifolds. 
\begin{theorem}\label{T1}
Let $M:=G/K$ be a compact  homogeneous manifold and let  $A=\text{Op}(\sigma)$ be a Fourier multiplier on $M$. If $A$ is bounded from $L^{p_1}(M)$ into $L^{p_2}(M),$ then $A$ extends to a bounded operator from $B^{r}_{p_{1},q}(M)$ into $B^{r}_{p_{2},q}(M),$ for all $r\in\mathbb{R},$  $1\leq p_1,p_2\leq \infty,$ and $0<q\leq \infty.$
\end{theorem}
As a consequence of this fact, we establish the following theorems. First, we present a theorem on the boundedness of operators on compact homogeneous spaces.
\begin{theorem}\label{T2}
Let us consider $A:C^{\infty}(G/K)\rightarrow C^{\infty}(G/K) $ be a pseudo-differential operator on the compact homogeneous manifold $G/K.$ Let $n=\dim(G/K)$ and $1<p_1\leq  2\leq p_2<\infty.$ Let us assume that the (global)  matrix valued symbol $a(x,\pi)$ of $A$ satisfies in terms of the Plancherel measure $\mu$ of $\widehat{G}_0$ the inequality,
\begin{equation}\label{RAN}
\sup_{s>0}s[\mu\{ \pi\in\widehat{G}_{0}: \Vert \partial_{x}^{\beta}a(x,\pi) \Vert_{op}>s\}]^{\frac{1}{p_1}-\frac{1}{p_2}}<\infty,
\end{equation}
for all $|\beta|\leq [\frac{n}{p_1}]+1.$ Then $A$ extends to a bounded operator from $B^{r}_{p_1,q}(G/K)$ into $B^{r}_{p_2,q}(G/K),$  for all $r\in\mathbb{R}$ and $0<q\leq \infty.$
\end{theorem}
It is important to mention that the Theorem above can be obtained as consequence of Theorem \ref{T1} and the results in \cite{RR}. 
\begin{remark}
A classical result by H\"ormander (Theorem 1.11 of \cite{Ho}) establish the boundedness of a Fourier multiplier of the form \eqref{pseudorn}  from $L^{p_1}(\mathbb{R}^n)$ into $L^{p_2}(\mathbb{R}^n)$ if its symbol $a(x,\xi):=a(\xi)$ satisfies the relation
\begin{equation}
\sup_{s>0}s[\mu\{ \xi\in\mathbb{R}^n: | a(\xi) |>s\}]^{\frac{1}{p_1}-\frac{1}{p_2}}<\infty,
\end{equation}
where $\mu$ is the Lebesgue measure and $1<p_1\leq 2\leq p_2<\infty.$ For compact homogeneous manifolds $M:=G/K,$ Ruzhansky, Akylzhanov and Nursultanov \cite{RR} have obtained the boundedness from $L^{p_1}(G)$ into $L^{p_2}(G)$ of  pseudo-differential operators $A$ with symbols satisfying the condition \eqref{RAN} and $1<p_1\leq 2\leq p_2<\infty$. (See also the references \cite{RR3} and \cite{RR2}). 
\end{remark}
Now for the case of compact Lie groups, which are important cases of homogeneous manifolds, we have the following  theorems on boundedness of operators associated to symbols satisfying conditions of H\"ormander type.
\begin{theorem}\label{T3}
Let $G$ be a compact Lie group, $n=\dim (G)$ and let $0\leq \rho,\delta \leq 1.$ Denote by $\varkappa$ the smallest even integer larger that $\frac{n}{2}.$ Let $1<p<\infty$ and $l=[\frac{n}{p}]+1. $ Let $A$ from $C^{\infty}(G)$ into $C^{\infty}(G)$ be a pseudo-differential operator with global symbol $a(x,\xi)$ satisfying
\begin{equation}\label{cardonacondition''}
\Vert \mathbb{D}_{\xi}^{\alpha}\partial_{x}^{\beta}a(x,\xi) \Vert_{op}\leq C_{\alpha,\beta}\langle \xi\rangle^{-m-\rho|\alpha|+\delta|\beta|}, \,\,|\alpha|\leq \varkappa,|\beta|\leq l,
\end{equation}
with $m\geq \varkappa(1-\rho)|\frac{1}{p}-\frac{1}{2}|+\delta l.$ Then $A$ extends to a bounded operator from $B^{r}_{p,q}(G)$ into $B^{r}_{p,q}(G)$ for all $0< q\leq \infty$ and $r\in\mathbb{R}.$ Moreover, if we assume that
\begin{equation}\label{cardonacondition'}
\Vert \mathbb{D}_{\xi}^{\alpha}\partial_{x}^{\beta}a(x,\xi) \Vert_{op}\leq C_{\alpha,\beta}\langle \xi\rangle^{-\rho|\alpha|}, \,\,|\alpha|\leq \varkappa,|\beta|\leq l,
\end{equation}  then $A$ extends to a bounded operator from $B^{r+ \varkappa(1-\rho)|\frac{1}{p}-\frac{1}{2}|  }_{p,q}(G)$  into $B^{r}_{p,q}(G)$  for all $1<p<\infty,$ $0<q\leq\infty$ and $r\in\mathbb{R}.$
\end{theorem}

\begin{theorem}\label{T4}
Let $G$ be a compact Lie group, $n=\dim (G),$ $0\leq \rho< 1$ and $0\leq \nu<\frac{n}{2}(1-\rho).$  Let $A$ from $C^{\infty}(G)$ into $C^{\infty}(G)$ be a pseudo-differential operator with global symbol $a(x,\xi)$ satisfying
\begin{equation}
\Vert \mathbb{D}_{\xi}^{\alpha}\partial_{x}^{\beta}a(x,\xi) \Vert_{op}\leq C_{\alpha,\beta}\langle \xi\rangle^{-\nu-\rho|\alpha|}, \,\,\alpha\in\mathbb{N}^{n},|\beta|\leq l,
\end{equation}
with $1<p<\infty$ and $l=[\frac{n}{p}]+1.$ Then $A$ extends to a bounded operator from $B^{r}_{p,q}(G)$ into $B^{r}_{p,q}(G)$ for all $p$ with $|\frac{1}{p}-\frac{1}{2}|\leq \frac{\nu}{n}(1-\rho)^{-1}, $ $0< q\leq \infty$ and $r\in\mathbb{R}.$
\end{theorem}
Now, we provide some remarks on our main results.  
\begin{itemize}
\item Theorems \ref{T3} and \eqref{T4} can be proved by using Theorem \eqref{T1}, and the $L^p$ boundedness theorems in \cite{DR4}. For the definition of the difference operators $\mathbb{D}_\xi^{\alpha},$  $\alpha\in\mathbb{N}^n$ (which were introduced in \cite{Ruz3}) we refer to Definition \ref{differenceoperators}.  
\item Recently in \cite{Cardona}, Theorem 1.6, the boundedness of pseudo-differential operators $A$ on every $B^{r}_{p,q}(G)$-space with symbols (of order zero) satisfying
\begin{equation}
\Vert \mathbb{D}_{\xi}^{\alpha}\partial_{x}^{\beta}a(x,\xi) \Vert_{op}\leq C_{\alpha,\beta}\langle \xi\rangle^{-|\alpha|}, \,\,|\alpha|\leq\varkappa,|\beta|\leq l,
\end{equation}
has been shown. This result has been obtained as consequence of the $L^p(G)$-boundedness of Fourier multipliers with symbols $a(\xi)$ satisfying the analogous condition
\begin{equation}
\Vert \mathbb{D}_{\xi}^{\alpha}a(\xi) \Vert_{op}\leq C_{\alpha}\langle \xi\rangle^{-|\alpha|}, \,\,|\alpha|\leq\varkappa.
\end{equation} 
\item In Theorem \ref{T3}, the condition \eqref{cardonacondition'}, generalizes the Theorem 1.6 of \cite{Cardona} (in fact, we only need to consider $\rho=1$). In the hypothesis \eqref{cardonacondition''} we do not use the usual conditions $\rho>\delta$ and $\rho \geq
1-\delta$ for the invariance under coordinates charts of the H\"ormander classes $\Psi^m_{\rho,\delta}(G)$ (see \cite{Ho}).

\item For operators with symbols whose derivatives $\mathbb{D}_{\xi}^{\alpha}a(\xi)$ are bounded by $C_{\alpha}\langle \xi\rangle^{-m-\rho|\alpha|}$ in the operator norm, $m\rightarrow 0^+,$ the $L^p$- boundedness  is valid only for finite intervals centered at $p=2,$ (c.f Delgado and Ruzhansky \cite{DR4}).  Since our Besov estimates are obtained from these $L^p$-estimates, we obtain the boundedeness of operators $A$ on $B^{r}_{p,q}(G)$ around of $p=2.$
\item It was proved by Fischer that for $\rho\geq 1-\delta,$ $\rho>\delta,$ an operator $A$ in the H\"ormander class $\Psi^m_{\rho,\delta}(G)$ has matrix valued symbol $a=\sigma_A$ satisfying
\begin{equation}
\Vert \mathbb{D}_{\xi}^{\alpha}\partial_{x}^{\beta}\sigma_A(x,\xi) \Vert_{op}\leq C_{\alpha,\beta}\langle \xi\rangle^{m-\rho|\alpha|+\delta|\beta|}, \,\,\alpha,\beta\in\mathbb{N}^n.
\end{equation}
The Bordaud result which asserts that every operator $A$ with symbol in $\Psi^m_{1,\delta}(\mathbb{R}^n),$ $0\leq \delta<1,$ is a bounded operator from $B^{r+m}_{p,q}(\mathbb{R}^n)$ into $ B^{r}_{p,q}(\mathbb{R}^n)$ implies the same result for the classes $\Psi^m_{1,\delta}(G)$. The boundedness on Besov spaces for operators associated to H\"ormander classes with $0<\rho<1$ has a different behavior to the case $\rho=1.$ In fact,  as consequence of the results in  Park \cite{BJpark}, an operator $A$ with symbol in $\Psi^m_{\rho,\delta}(\mathbb{R}^n),$ $\rho\geq \delta,$ is bounded from $B^s_{p,q}(\mathbb{R}^n)$ into $B^r_{p,q}(\mathbb{R}^n)$ under the condition $s=m+r+n(1-\rho)|\frac{1}{p}-\frac{1}{2}|$ (notice that for $\rho=1$ this result is nothing else that the Bordaud theorem). In particular if $m=0,$ then $A:B^{r+n(1-\rho)|\frac{1}{p}-\frac{1}{2}|}_{p,q}(\mathbb{R}^n)\rightarrow B^r_{p,q}(\mathbb{R}^n)$ is bounded. Again, the Fischer result mentioned above together with the Park result implies for $\rho\geq 1-\delta,$ $\rho>\delta,$ $\sigma_A\in \Psi_{\rho,\delta}^m(G)$  the boundedness of $A:B^{r+n(1-\rho)|\frac{1}{p}-\frac{1}{2}|}_{p,q}(G)\rightarrow B^r_{p,q}(G),$ for all $r\in\mathbb{R}$. The novelty of  our results is that we consider matrix valued symbols of limited regularity, and we do not impose the condition $\rho\geq 1-\delta,$ $\rho>\delta.$ 
\item Although in our main results we consider symbols with  order less than or equal to zero, these results can be extended to symbols of arbitrary order by using standard techniques. In fact, if $A:C^{\infty}(G)\rightarrow C^\infty(G)$ is a linear and bounded operator, then under any one of the following conditions
\begin{itemize}
\item $ \Vert  \mathbb{D}_{\xi}^{\alpha}\partial_{x}^\beta\sigma_{A}(x,\xi)\Vert_{op}\leq C_{\alpha}\langle \xi \rangle^{m-|\alpha|}, \text{ for all } \,\,\,\,|\alpha|\leq \varkappa,\,\,\,\,\,|\beta|\leq l,$
\item $\Vert \mathbb{D}_{\xi}^{\alpha}\partial_{x}^{\beta}\sigma_A(x,\xi) \Vert_{op}\leq C_{\alpha,\beta}\langle \xi\rangle^{m-\nu-\rho|\alpha|}, \,\,\alpha\in\mathbb{N}^{n},|\beta|\leq l,\,\,\,|\frac{1}{p}-\frac{1}{2}|\leq \frac{\nu}{n}(1-\rho)^{-1},\,\,0\leq \nu<\frac{n}{2}(1-\rho),$
\item $\Vert \mathbb{D}_{\xi}^{\alpha}\partial_{x}^{\beta}a(x,\xi) \Vert_{op}\leq C_{\alpha,\beta}\langle \xi\rangle^{m-\varkappa(1-\rho)|\frac{1}{p}-\frac{1}{2}|-\delta l-\rho|\alpha|+\delta|\beta|}, \,\,|\alpha|\leq \varkappa,|\beta|\leq l,$
\end{itemize}
the corresponding pseudo-differential operator $A$ extends to a bounded operator from $B^{r+m}_{p,q}(G)$ into $B^{r}_{p,q}(G)$ for all $r\in \mathbb{R},$ $1<p<\infty$ and $0<q\leq \infty.$ Additionally, we observe that the condition
\begin{itemize}
\item $
\Vert \mathbb{D}_{\xi}^{\alpha}\partial_{x}^{\beta}a(x,\xi) \Vert_{op}\leq C_{\alpha,\beta}\langle \xi\rangle^{m-\rho|\alpha|}, \,\,|\alpha|\leq \varkappa,|\beta|\leq l,
$
\end{itemize} assures that $A$ extends to a bounded operator from $B^{m+r+ \varkappa(1-\rho)|\frac{1}{p}-\frac{1}{2}|  }_{p,q}(G)$  into $B^{r}_{p,q}(G)$  for all $1<p<\infty,$ $0<q\leq\infty$ and $r\in\mathbb{R}$ similar to the Park result for euclidean symbols.
\end{itemize}


\section{Pseudo-differential operators on compact Lie groups}\label{Preliminaries}
\subsection{Fourier analysis and Sobolev spaces on compact Lie groups } In this section we will introduce some preliminaries on pseudo-differential operators on compact Lie groups and some of its properties on $L^p$-spaces. There are two notions of pseudo-differential operators on compact Lie groups. The first notion in the case of general manifolds (based on the idea of {\em local symbols}) and, in a much more recent context, the one of global pseudo-differential operators on compact Lie groups as defined by Ruzhansky and Turunen \cite{Ruz} (see also \cite{RuzTuru2}). We adopt this last notion for our work. We will always equip a compact Lie group with the Haar measure $\mu_{G}.$ For simplicity, we will write $\int_{G}fdx$ for $\int_{G}f(x) d\mu_{G}(x),$ $L^p(G)$ for $L^p(G,\mu_{G}),$ etc.  The following assumptions are based on the group Fourier transform
 $$  \widehat{\varphi}(\xi)=\int_{G}\varphi(x)\xi(x)^*dx,\,\,\,\,\,\,\,\,\,\,\,\,\,\,\,\, \varphi(x)=\sum_{[\xi]\in \widehat{G}}d_{\xi}\text{Tr}(\xi(x)\widehat{\varphi}(\xi)) .$$
The Peter-Weyl Theorem on $G$ implies  the Plancherel identity on $L^2(G),$
$$ \Vert \varphi \Vert_{L^2(G)}= \left(\sum_{[\xi]\in \widehat{G}}d_{\xi}\text{Tr}(\widehat{\varphi}(\xi)\widehat{\varphi}(\xi)^*) \right)^{\frac{1}{2}}=\Vert  \widehat{\varphi}\Vert_{ L^2(\widehat{G} ) } .$$
\noindent Here  $$\Vert A \Vert_{HS}=\text{Tr}(AA^*),$$ denotes the Hilbert-Schmidt norm of matrices. Any linear operator $A$ on $G$ mapping $C^{\infty}(G)$ into $\mathcal{D}'(G)$ gives rise to a {\em matrix-valued global (or full) symbol} $\sigma_{A}(x,\xi)\in \mathbb{C}^{d_\xi \times d_\xi}$ given by
\begin{equation}
\sigma_A(x,\xi)=\xi(x)^{*}(A\xi)(x),
\end{equation}
which can be understood from the distributional viewpoint. Then it can be shown that the operator $A$ can be expressed in terms of such a symbol as \cite{Ruz}
\begin{equation}Af(x)=\sum_{[\xi]\in \widehat{G}}d_{\xi}\text{Tr}[\xi(x)\sigma_A(x,\xi)\widehat{f}(\xi)]. 
\end{equation}
In this paper we use the notation $\text{Op}(\sigma_A)=A.$ $L^p(\widehat{G})$ spaces on the unitary dual can be well defined.   If $p=2,$ $L^2(\widehat{G})$ is defined by the norm  $$\Vert \Gamma \Vert^2_{L^2(\widehat{G})}=\sum_{[\xi]\in\widehat{G}}d_{\xi}\Vert \Gamma (\xi)\Vert^2_{HS}.$$ 
Now, we want to introduce Sobolev spaces and, for this, we give some basic tools. \noindent Let $\xi\in \textnormal{Rep}(G):=\cup \widehat{G}=\{\xi:[\xi]\in\widehat{G}\},$ if $x\in G$ is fixed, $\xi(x):H_{\xi}\rightarrow H_{\xi}$ is an unitary operator and $d_{\xi}:=\dim H_{\xi} <\infty.$ There exists a non-negative real number $\lambda_{[\xi]}$ depending only on the equivalence class $[\xi]\in \hat{G},$ but not on the representation $\xi,$ such that $-\mathcal{L}_{G}\xi(x)=\lambda_{[\xi]}\xi(x);$ here $\mathcal{L}_{G}$ is the Laplacian on the group $G$ (in this case, defined as the Casimir element on $G$). Let  $\langle \xi\rangle$ denote the function $\langle \xi \rangle=(1+\lambda_{[\xi]})^{\frac{1}{2}}$.  
\begin{definition}\label{sov} For every $s\in\mathbb{R},$ the {\em Sobolev space} $H^s(G)$ on the Lie group $G$ is  defined by the condition: $f\in H^s(G)$ if only if $\langle \xi \rangle^s\widehat{f}\in L^{2}(\widehat{G})$. The Sobolev space $H^{s}(G)$ is a Hilbert space endowed with the inner product $\langle f,g\rangle_{s}=\langle \Lambda_{s}f, \Lambda_{s}g\rangle_{L^{2}(G)}$, where, for every $r\in\mathbb{R}$, $\Lambda_{s}:H^r\rightarrow H^{r-s}$ is the bounded pseudo-differential operator with symbol $\langle \xi\rangle^{s}I_{\xi}.$ In $L^p$ spaces,  the $p$-Sobolev space of order $s,$ $H^{s,p}(G),$ is defined by functions satisfying
\begin{equation}
\Vert f \Vert_{H^{s,p}(G)}:=\Vert \Lambda_sf\Vert_{L^p(G)}<\infty.
\end{equation}
\end{definition}
\subsection{Differential and difference operators} In order to classify symbols by its regularity we present the usual definition of differential operators and the difference operators used introduced in \cite{Ruz3}.
\begin{definition}\label{differenceoperators} Let $(Y_{j})_{j=1}^{\text{dim}(G)}$ be a basis for the Lie algebra $\mathfrak{g}$ of $G$, and let $\partial_{j}$ be the left-invariant vector fields corresponding to $Y_j$. We define the differential operator associated to such a basis by $D_{Y_j} = \partial_{j}$ and, for every $\alpha\in\mathbb{N}^n$, the {\em differential operator} $\partial^{\alpha}_{x}$ is the one given by $\partial_x^{\alpha}=\partial_{1}^{\alpha_1}\cdots \partial_{n}^{\alpha_n}$. Now, if $\xi_{0}$ is a fixed irreducible  representation, the matrix-valued {\em difference operator} is the given by $\mathbb{D}_{\xi_0}=(\mathbb{D}_{\xi_0,i,j})_{i,j=1}^{d_{\xi_0}}=\xi_{0}(\cdot)-I_{d_{\xi_0}}$. If the representation is fixed we omit the index $\xi_0$ so that, from a sequence $\mathbb{D}_1=\mathbb{D}_{\xi_0,j_1,i_1},\cdots, \mathbb{D}_n=\mathbb{D}_{\xi_0,j_n,i_n}$ of operators of this type we define $\mathbb{D}^{\alpha}_{\xi}=\mathbb{D}_{1}^{\alpha_1}\cdots \mathbb{D}^{\alpha_n}_n$, where $\alpha\in\mathbb{N}^n$.
\end{definition}
\subsection{Besov spaces} We introduce the Besov spaces on compact Lie groups using the  Fourier transform on the group $G$ as follow.
\begin{definition}\label{besovspaces} Let $r\in\mathbb{R},$ $0\leq q<\infty$ and $0<p\leq \infty.$ If $f$ is a measurable function on $G,$ we say that $f\in B^r_{p,q}(G)$ if $f$ satisfies
\begin{equation}\label{n1}\Vert f \Vert_{B^r_{p,q}}:=\left( \sum_{m=0}^{\infty} 2^{mrq}\Vert \sum_{2^m\leq \langle \xi\rangle< 2^{m+1}}  d_{\xi}\text{Tr}[\xi(x)\widehat{f}(\xi)]\Vert^q_{L^p(G)}\right)^{\frac{1}{q}}<\infty.
\end{equation}
If $q=\infty,$ $B^r_{p,\infty}(G)$ consists of those functions $f$ satisfying
\begin{equation}\label{n2}\Vert f \Vert_{B^r_{p,\infty}}:=\sup_{m\in\mathbb{N}} 2^{mr}\Vert \sum_{2^m\leq \langle \xi\rangle < 2^{m+1}}  d_{\xi}\text{Tr}[\xi(x)\widehat{f}(\xi)]\Vert_{L^p(G)}<\infty.
\end{equation}
\end{definition}
If we denote by $\textnormal{Op}(\chi_{m})$ the Fourier multiplier associated to the symbol
$$ \chi_{m}(\eta)=1_{\{[\xi]:2^m\leq \langle \xi\rangle< 2^{m+1}\}}(\eta), $$ we also write,
\begin{equation}
\Vert f\Vert_{B^r_{p,q}}=\Vert \{2^{mr} \Vert \textnormal{Op}(\chi_{m})f \Vert_{L^p(G)}\}_{m=0}^{\infty} \Vert_{l^q(\mathbb{N})},\,\,0<p,q\leq \infty,\,r\in\mathbb{R}.
\end{equation}
\begin{remark}
For every $s\in\mathbb{R},$ $H^{s,2}(G)=H^{s}(G)=B^{r}_{2,2}(G).$ Besov spaces according to Definition \eqref{besovspaces} were introduced in \cite{RuzBesov} on compact homogeneous manifolds where, in particular, the authors obtained its embedding properties. On compact Lie groups such spaces where characterized, via representation theory, in \cite{RuzBesov2}.
\end{remark}
\begin{remark}
In connection with our  comments in the introduction, a detailed description on euclidean models and the role of Besov spaces in the context of applications we refer the reader to the book of A. Cohen \cite{Cohen}. The reference Hairer, \cite{Hairer} explains the importance of the Besov spaces in the setting of the theory of regularity structures as well as a theorem of reconstruction and some interactions with stochastic partial differential equations; on the other hand, as it was pointed out in 
 in \cite{FE} (see also references therein) several problems in signal analysis and information theory
require non-euclidean models. These models include: spheres, projective spaces and
general compact manifolds, hyperboloids and general non-compact symmetric spaces, and finally various Lie groups. In connection with these spaces it is important to study Besov spaces on compact and non-compact manifolds.
\end{remark}
\subsection{Global operators on compact homogeneous manifolds in Lebesgue spaces} Now we introduce the notion of homogeneous manifold. Let us consider a closed subgroup $K$ of $G$ and identify $M:=G/K$ as a analytic manifold in a canonical way. (In the case $K=\{e\} $ where $e$ is the identity element of the group $G,$ we identify $G/K$ with $G$). Let us denote by $\widehat{G}_0$ the subset of $\widehat{G}$ that are representations of type I with respect to the subgroup $K.$ This means that $\pi(h)(a)=a$ for all $h\in K.$ Besov spaces on homogeneous manifolds $M=G/K$ can be defined, and the Besov norms are defined as in \eqref{n1} y \eqref{n2}, but the representations $\xi$ in the sums are in $\widehat{G}_0.$
The following $L^{p}-L^{q}$-theorem will be useful in our analysis of Besov continuity for pseudo-differential operators on homogeneous manifolds (c.f. \cite{RR}).
\begin{theorem}\label{t1}
Let us consider $A:C^{\infty}(G/K)\rightarrow C^{\infty}(G/K) $ be a pseudo-differential operator operator on the compact homogeneous manifold $G/K.$ Let $n=\dim(G/K)$ and $1<p_1\leq 2\leq p_2<\infty.$ Let us assume that the (global) symbol matrix valued $a(x,\pi)$ of $A$ satisfies
\begin{equation}
\sup_{s>0}s[\mu\{ \pi\in\widehat{G}_{0}: \Vert \partial_{x}^{\beta}a(x,\pi) \Vert_{op}>s\}]^{\frac{1}{p_1}-\frac{1}{p_2}}<\infty,
\end{equation}
for all $|\beta|\leq [\frac{n}{p_1}]+1.$ Then $A$ extends to a bounded operator from $L^{p_1}(G)$ into $L^{p_2}(G).$
\end{theorem}
The following sharp $L^p$ theorem on $G$ allow us to investigate Besov continuity for pseudo-differential operators on compact Lie groups. (c.f. Delgado and Ruzhansky \cite{DR4}).
\begin{theorem}\label{t2}
Let $G$ be a Compact Lie group, $n=\dim (G)$ and let $0\leq \rho,\leq \delta \leq 1.$ Denote by $\varkappa$ the smallest even integer larger that $\frac{n}{2}.$ Let $1<p<\infty$ and $l=[\frac{n}{p}]+1. $ Let $A:C^{\infty}(G)$ into $C^{\infty}(G)$ be a pseudo-differential operator with global symbol $a(x,\xi)$ satisfying
\begin{equation}
\Vert \mathbb{D}_{\xi}^{\alpha}\partial_{x}^{\beta}a(x,\xi) \Vert_{op}\leq C_{\alpha,\beta}\langle \xi\rangle^{-m-\rho|\alpha|+\delta|\beta|}, \,\,|\alpha|\leq \varkappa,|\beta|\leq l,
\end{equation}
with $m\geq \varkappa(1-\rho)|\frac{1}{p}-\frac{1}{2}|+\delta l.$ Then $A$ extends to a bounded operator from $L^{p}(G)$ into $L^{p}(G).$ 
\end{theorem}

\begin{theorem}\label{t3}
Let $G$ be a compact Lie group, $n=\dim (G),$ $0\leq \rho< 1$ and $0\leq \nu<\frac{n}{2}(1-\rho).$  Let $A:C^{\infty}(G)$ into $C^{\infty}(G)$ be a pseudo-differential operator with global symbol $a(x,\xi)$ satisfying
\begin{equation}
\Vert \mathbb{D}_{\xi}^{\alpha}\partial_{x}^{\beta}\sigma(x,\xi) \Vert_{op}\leq C_{\alpha,\beta}\langle \xi\rangle^{-\nu-\rho|\alpha|+\delta|\beta|}, \,\,\alpha\in\mathbb{N}^{n},|\beta|\leq l,
\end{equation}
with $1<p<\infty$ and $l=[\frac{n}{p}]+1.$ Then $A$ extends to a bounded operator from $L^{p}(G)$ into $L^{p}(G)$  for all $p$ with  $|\frac{1}{p}-\frac{1}{2}|\leq \frac{\nu}{n}(1-\rho)^{-1}.$ 
\end{theorem}

\section{Global pseudo-differential operators in Besov spaces}\label{proof}
 In this section we prove our main results. For the case of compact Lie groups, 
Our starting point is the following theorem, which gives a relationship between $L^{p}$ boundedness and Besov continuity on homogeneous compact manifolds. A Fourier multiplier on $M:=G/K$ is an operator $A=\mathrm{Op}(\sigma)$ with symbol $\sigma(\xi)$ satisfying $\sigma(\xi)_{ij}=0$ for $i>k_\xi$ or $j>k_\xi,$ $[\xi]\in \widehat{G}_0.$
\begin{theorem}
Let $M:=G/K$ be a compact homogeneous manifold and let  $A=\text{Op}(\sigma)$ be a Fourier multiplier on $M$. If $A$ is bounded from $L^{p_1}(M)$ into $L^{p_2}(M),$ then $A$ extends to a bounded operator from $B^{r}_{p_{1},q}(M)$ into $B^{r}_{p_{2},q}(M),$ for all $r\in\mathbb{R},$ $1\leq p_1,p_2\leq \infty,$ and $0<q\leq \infty.$
\end{theorem}
\begin{proof}
First, let us consider a multiplier operator $\text{Op}(\sigma)$ bounded from $L^{p_1}(M)$ into $L^{p_2}(M),$ and $f\in C^{\infty}(M).$ Then, we have
$ \Vert Tf\Vert_{L^{p_2}(M)} \leq C\Vert f \Vert_{L^{p_1}(M)},$
where $C=\Vert T\Vert_{B(L^{p_1},L^{p_2})}$ is the usual operator norm. We denote by $\chi_{m}(\xi)$ the characteristic function of $D_{m}:=\{ \xi\in \widehat{G}_0: 2^{m}\leq \langle \xi\rangle<2^{m+1} \}$ and $\text{Op}(\chi_{m})$ the corresponding Fourier multiplier of the symbol $\chi_{m}(\xi)I_{\xi}$. Here, $I_{\xi}:=(a_{ij})$ is the matrix in $\mathbb{C}^{d_{\xi}\times d_{\xi}},$ defined by $a_{ii}=1$ if $1\leq i\leq k_\xi$ and $a_{ij}=0$ in other case. By the definition of Besov norm, if $0<q<\infty$ we have
\begin{align*}
\Vert \text{Op}(\sigma)f \Vert^{q}_{{B^r}_{p_2,q}} &= \sum_{m=0}^{\infty} 2^{mrq}\Vert \sum_{2^m\leq \langle \xi\rangle< 2^{m+1}}  d_{\xi}\text{Tr}[\xi(x)\mathscr{F}(\text{Op}(\sigma)f)(\xi)]\Vert^q_{L^{p_2}(M)}\\
&= \sum_{m=0}^{\infty} 2^{mrq}\Vert \sum_{[\xi]\in\widehat{G}_0}  d_{\xi}\cdot \chi_{m}(\xi)\text{Tr}[\xi(x)\mathscr{F}(\text{Op}(\sigma)f)(\xi)]\Vert^q_{L^{p_2}(M)}\\
&= \sum_{m=0}^{\infty} 2^{mrq}\Vert \sum_{[\xi]\in\widehat{G}_0}  d_{\xi}\cdot \text{Tr}[\xi(x)\chi_{m}(\xi)\sigma(\xi)(\mathscr{F}f)(\xi)]\Vert^q_{L^{p_2}(M)}\\
&= \sum_{m=0}^{\infty} 2^{mrq}\Vert \sum_{[\xi]\in\widehat{G}_0}  d_{\xi}\cdot \text{Tr}[\xi(x)\sigma(\xi)\mathscr{F}(\text{Op}(\chi_{m})f)(\xi)]\Vert^q_{L^{p_2}(M)}\\
&= \sum_{m=0}^{\infty} 2^{mrq}\Vert \text{Op}(\sigma)[ (\text{Op}(\chi_{m})f)]\Vert^q_{L^{p_2}(M)}.\\
\end{align*}
By the boundedness of $\text{Op}(\sigma)$ from $L^{p_1}(M)$ into $L^{p_2}(M)$ we get,
\begin{align*}
\Vert \text{Op}(\sigma)f \Vert^{q}_{{B^r}_{p_2,q}}&\leq \sum_{m=0}^{\infty} 2^{mrq}C^q\Vert  \text{Op}(\chi_{m})f\Vert^q_{L^{p_1}(M)}\\
&= \sum_{m=0}^{\infty} 2^{mrq}C^q\Vert \sum_{[\xi]\in\widehat{G}_0}  d_{\xi}\cdot \text{Tr}[\xi(x)\chi_{m}(\xi)I_{\xi}\mathscr{F}(f)(\xi)]\Vert^q_{L^{p_1}(M)}\\
&= \sum_{m=0}^{\infty} 2^{mrq}C^q\Vert \sum_{2^m\leq \langle \xi\rangle<2^{m+1}}  d_{\xi}\cdot \text{Tr}[\xi(x)I_{\xi}\mathscr{F}(f)(\xi)]\Vert^q_{L^{p_1}(M)}\\
&= \sum_{m=0}^{\infty} 2^{mrq}C^q\Vert \sum_{2^m\leq \langle \xi\rangle<2^{m+1}}  d_{\xi}\cdot \text{Tr}[\xi(x)I_{\xi}\mathscr{F}(f)(\xi)]\Vert^q_{L^{p_1}(M)}\\
&=C^q\Vert f \Vert^{q}_{{B^r}_{p_1,q}}
\end{align*}
Hence,
\begin{align*}
\Vert \text{Op}(\sigma)f \Vert_{{B^r}_{p_2,q}} \leq C\Vert f \Vert_{{B^r}_{p_1,q}}.
\end{align*}
If $q=\infty$  we have
\begin{align*}
\Vert \text{Op}(\sigma)f \Vert_{{B^r}_{p_2,\infty}} &= \sup_{m\in\mathbb{N}} 2^{mr}\Vert \sum_{2^m\leq \langle \xi\rangle< 2^{m+1}}  d_{\xi}\text{Tr}[\xi(x)\mathscr{F}(\text{Op}(\sigma)f)(\xi)]\Vert_{L^{p_2}(M)}\\
&= \sup_{m\in\mathbb{N}} 2^{mr}\Vert \sum_{[\xi]\in\widehat{G}}  d_{\xi}\cdot \chi_{m}(\xi)\text{Tr}[\xi(x)\mathscr{F}(\text{Op}(\sigma)f)(\xi)]\Vert_{L^{p_2}(M)}\\
&= \sup_{m\in\mathbb{N}} 2^{mr}\Vert \sum_{[\xi]\in\widehat{G}}  d_{\xi}\cdot \text{Tr}[\xi(x)\chi_{m}(\xi)\sigma(\xi)(\mathscr{F}f)(\xi)]\Vert_{L^{p_2}(M)}\\
&=\sup_{m\in\mathbb{N}} 2^{mr}\Vert \sum_{[\xi]\in\widehat{G}}  d_{\xi}\cdot \text{Tr}[\xi(x)\sigma(\xi)\mathscr{F}(\text{Op}(\chi_{m})f)(\xi)]\Vert_{L^{p_2}(M)}\\
&=\sup_{m\in\mathbb{N}} 2^{mr}\Vert \text{Op}(\sigma)[ (\text{Op}(\chi_{m})f)]\Vert_{L^{p_2}(M)}.\\
\end{align*}
Newly, by using the fact that $\text{Op}(\sigma)$ is a bounded operator from $L^{p_1}(M)$ into $L^{p_2}(M)$ we have,
\begin{align*}
\Vert \text{Op}(\sigma)f \Vert_{{B^r}_{p_2,\infty}} &\leq \sup_{m\in\mathbb{N}} 2^{mr}C\Vert  \text{Op}(\chi_{m})f\Vert_{L^{p_1}(M)}\\
&= \sup_{m\in\mathbb{N}} 2^{mr}C\Vert \sum_{[\xi]\in\widehat{G}}  d_{\xi}\cdot \text{Tr}[\xi(x)\chi_{m}(\xi)I_{\xi}\mathscr{F}(f)(\xi)]\Vert_{L^{p_1}(M)}\\
&= \sup_{m\in\mathbb{N}} 2^{mr}C\Vert \sum_{2^m\leq \langle \xi\rangle<2^{m+1}}  d_{\xi}\cdot \text{Tr}[\xi(x)I_{\xi}\mathscr{F}(f)(\xi)]\Vert_{L^{p_1}(M)}\\
&= \sup_{m\in\mathbb{N}} 2^{mr}C\Vert \sum_{2^m\leq \langle \xi\rangle<2^{m+1}}  d_{\xi}\cdot \text{Tr}[\xi(x)I_{\xi}\mathscr{F}(f)(\xi)]\Vert_{L^{p_1}(M)}\\
&=C\Vert f \Vert_{{B^r}_{p_1,\infty}}.
\end{align*}
This implies that,
\begin{align*}
\Vert \text{Op}(\sigma)f \Vert_{{B^r}_{p_2,\infty}} \leq C\Vert f \Vert_{{B^r}_{p_1,\infty}}.
\end{align*}
With the last inequality we end the proof.
\end{proof}

\begin{theorem}
Let us consider $A:C^{\infty}(G/K)\rightarrow C^{\infty}(G/K) $ be a pseudo-differential operator operator on the compact homogeneous manifold $G/K.$ Let $n=\dim(G/K)$ and $1<p_1\leq 2\leq p_2<\infty.$ Let us assume that the (global)  matrix valued symbol $a(x,\pi)$ of $A$ satisfies in terms of the Plancherel measure $\mu$ on $\widehat{G}_0$ the inequality,
\begin{equation}
\sup_{s>0}s[\mu\{ \pi\in\widehat{G}_{0}: \Vert \partial_{x}^{\beta}a(x,\pi) \Vert_{op}>s\}]^{\frac{1}{p_1}-\frac{1}{p_2}}<\infty,
\end{equation}
for all $|\beta|\leq [\frac{n}{p_1}]+1.$ Then $A$ extends to a bounded operator from $B^{r}_{p_1,q}(G/K)$ into $B^{r}_{p_2,q}(G/K),$  for all $r\in\mathbb{R}$ and $0<q\leq \infty.$
\end{theorem}
\begin{proof}
If we assume that $A$ has symbol $\sigma(x,\pi)=\sigma(\pi)$ independent of $x\in M= G/K,$ then by Theorem \ref{t1} we have that $A$ is bounded from $L^{p_1}(M)$ into $L^{p_2}(M).$ By Theorem \ref{T1}, $A$ extends to a bounded operator from $B^{r}_{p_1,q}(M)$ into $B^{r}_{p_2,q}(M),$  for all $r\in\mathbb{R}$ and $0<q\leq \infty.$ Next, we consider the general case where $a(x,\pi)$ depends on $x.$  To do this we write for $f\in C^{\infty}(M):$
\begin{align*} Af(x)&=\sum_{[\xi]\in\widehat{G}_0}d_{\xi} \text{Tr}[\xi(x)\sigma(x,\xi)\widehat{f}(\xi)]\\
&=\int_{M }[\sum_{[\xi]\in\widehat{G}_0}d_{\xi} \text{Tr}[ \xi(y^{-1}x)\sigma(x,\xi)]]f(y)dy\\
&=\int_{M }[\sum_{[\xi]\in\widehat{G}_0}d_{\xi} \text{Tr}[ \xi(y)\sigma(x,\xi)]]f(xy^{-1})dy.\\
\end{align*}
Hence $A=\mathrm{Op}(\sigma)f(x)=(\kappa(x,\cdot)\ast f)(x),$ where
\begin{equation} \kappa(z,y)=\sum_{[\xi]\in\widehat{G}_0}d_{\xi} \text{Tr}[ \xi(y)\sigma(z,\xi)],
\end{equation}
and $\ast$ is the right convolution operator. Moreover, if we define $A_{z}f(x)=(\kappa(z,\cdot)\ast f)(x)$ for every element  $z\in M,$ we have $$A_{x}f(x)=Af(x),\,\,\,x\in M.$$ For all $0\leq |\beta|\leq [n/p]+1$ we have $\partial^{\beta}_{z}A_{z}f(x)=\text{Op}(\partial_{z}^{\beta}\sigma(z,\cdot))f(x).$ So, by the precedent argument on Fourier multipliers, for every $z\in M,$ $\partial^{\beta}_{z}A_{z}f=\text{Op}(\partial_{z}^{\beta}\sigma(z,\cdot))f$ is a bounded operator from $B^{r}_{p_1,q}(M)$ into $B^{r}_{p_2,q}$ for all  $r\in\mathbb{R}$ and $0<q\leq \infty.$ Now, we want to estimate the Besov norm of $\text{Op}(\sigma(\cdot,\cdot)).$ First, we observe that
\begin{align*}
&\left\Vert \sum_{2^m\leq \langle \xi\rangle <2^{m+1}} d_{\xi}\text{Tr}[\xi(x)\mathscr{F}(\text{Op}(\sigma)f)(\xi)] \right\Vert^p_{L^{p_2}}\\&:=\int_{M}\left|   \sum_{2^m\leq \langle \xi\rangle <2^{m+1}} d_{\xi} \text{Tr}[\xi(x)\int_{M}\text{Op}(\sigma)f(y)\xi(y)^{*}dy] \right|^{p_2}dx\\
&=\int_{M}\left|   \sum_{2^m\leq \langle \xi\rangle <2^{m+1}} d_{\xi} \text{Tr}[\xi(x)\int_{M}A_{y}f(y)\xi(y)^{*}dy] \right|^{p_2}dx\\
&\leq\int_{M}\sup_{z\in M}\left|   \sum_{2^m\leq \langle \xi\rangle <2^{m+1}} d_{\xi} \text{Tr}[\xi(x)\mathscr{F}(A_{z}f)(\xi)] \right|^{p_2}dx\\
\end{align*}
By the Sobolev embedding theorem, we have
\begin{align*}
&\sup_{z\in M} \left|  \sum_{2^m\leq \langle \xi\rangle <2^{m+1}} d_{\xi} \text{Tr}[\xi(x)\mathscr{F}(A_{z}f)(\xi)] \right|^{p_2}\\
&\lesssim  \sum_{|\beta|\leq l} \int_{M} \left|  \sum_{2^m\leq \langle \xi\rangle <2^{m+1}} d_{\xi} \text{Tr}[\xi(x)\mathscr{F}(\text{Op}(\partial^{\beta}_{z}\sigma(z,\cdot))f)(\xi)] \right|^{p_2} dz \\
&\lesssim  \sup_{|\beta|\leq l} \int_{M} \left|  \sum_{2^m\leq \langle \xi\rangle <2^{m+1}} d_{\xi} \text{Tr}[\xi(x)\mathscr{F}(\text{Op}(\partial^{\beta}_{z}\sigma(z,\cdot))f)(\xi)] \right|^{p_2} dz
\end{align*}
From this, and the Sobolev embedding theorem we have
\begin{align*}
&\int_{M} \sup_{z\in G} \left|  \sum_{2^m\leq \langle \xi\rangle <2^{m+1}} d_{\xi} \text{Tr}[\xi(x)A_{z}f(y)\xi(y)^{*}] \right|^pdy\\
&\lesssim \sum_{|\beta|\leq l}\int_{M}\int_{M} \left|  \sum_{2^m\leq \langle \xi\rangle <2^{m+1}} d_{\xi} \text{Tr}[\xi(x)\mathscr{F}(\text{Op}(\partial^{\beta}_{z}\sigma(z,\cdot))f)(\xi)] \right|^{p_2} dz dx\\
& \lesssim\sup_{|\beta|\leq l}\int_{M}\int_{M} \left|  \sum_{2^m\leq \langle \xi\rangle <2^{m+1}} d_{\xi} \text{Tr}[\xi(x)\mathscr{F}(\text{Op}(\partial^{\beta}_{z}\sigma(z,\cdot))f)(\xi)] \right|^{p_2} dx dz\\
&\leq \sup_{|\beta|\leq l,z\in M}\int_{M} \left|  \sum_{2^m\leq \langle \xi\rangle <2^{m+1}} d_{\xi} \text{Tr}[\xi(x)\mathscr{F}(\text{Op}(\partial^{\beta}_{z}\sigma(z,\cdot))f)(\xi)] \right|^{p_2} dx \\
&= \sup_{|\beta|\leq l,z\in M}\Vert \sum_{2^m\leq \langle \xi\rangle <2^{m+1}} d_{\xi} \text{Tr}[\xi(x)\mathscr{F}(\text{Op}(\partial^{\beta}_{z}\sigma(z,\cdot))f)(\xi)]   \Vert^{p_2}_{L^{p_2}}
\end{align*}
Hence,
\begin{align*} &\Vert \sum_{2^m\leq \langle \xi\rangle <2^{m+1}} d_{\xi}\text{Tr}[\xi(x)\mathscr{F}(\text{Op}(\sigma)f)(\xi)] \Vert_{L^{p_2}}\\
&\lesssim \sup_{|\beta|\leq l,z\in M}\Vert \sum_{2^m\leq \langle \xi\rangle <2^{m+1}} d_{\xi} \text{Tr}[\xi(x)\mathscr{F}(\text{Op}(\partial^{\beta}_{z}\sigma(z,\cdot))f)(\xi)]   \Vert_{L^{p_2}}
\end{align*}
Thus, considering $0<q<\infty$ we obtain
\begin{align*}
\Vert\text{Op}(\sigma)f \Vert_{B^{r}_{p_2,q}({M})}&:=\left( \sum_{m=0}^{\infty} 2^{mrq}\left\Vert \sum_{2^m\leq \langle \xi\rangle <2^{m+1}} d_{\xi}\text{Tr}[\xi(x)\mathscr{F}(\text{Op}(\sigma)f)(\xi)] \right\Vert_{L^{p_2}}^q\right)^{\frac{1}{q}}\\
&\lesssim \left( \sum_{m=0}^{\infty} 2^{mrq} \sup_{|\beta|\leq l,z\in M}\left\Vert \sum_{2^m\leq \langle \xi\rangle <2^{m+1}} d_{\xi} \text{Tr}[\xi(x)\mathscr{F}(\text{Op}(\partial^{\beta}_{z}\sigma(z,\cdot))f)(\xi)]   \right\Vert_{L^{p_2}}^q   \right)^{\frac{1}{q}}
\end{align*}
We define for every $z\in M$ the non-negative function $z\mapsto g(z)$ by
$$ g(z)= \sup_{|\beta|\leq l}\left\Vert \sum_{2^m\leq \langle \xi\rangle <2^{m+1}} d_{\xi} \text{Tr}[\xi(x)\mathscr{F}(\text{Op}(\partial^{\beta}_{z}\sigma(z,\cdot))f)(\xi)]   \right\Vert_{L^{p_2}}^q .$$
We write,
\begin{align*}
\left( \sum_{m=0}^{\infty} 2^{mrq}  \sup_{z\in M }g(z)   \right)^{\frac{1}{q}} &= \lim_{k\rightarrow\infty}\left( \sum_{m=0}^{k} 2^{mrq}   \sup_{z\in M }g(z)    \right)^{\frac{1}{q}}=\lim_{k\rightarrow\infty}\left( \sup_{z\in M} \sum_{m=0}^{k} 2^{mrq}  g(z)   \right)^{\frac{1}{q}}\\
&\leq \lim_{k\rightarrow\infty}\left( \sup_{z\in M} \sum_{m=0}^{\infty} 2^{mrq}  g(z)   \right)^{\frac{1}{q}}=\sup_{z\in M}\left(  \sum_{m=0}^{\infty} 2^{mrq}  g(z)   \right)^{\frac{1}{q}}.
\end{align*}
Hence, we can write
\begin{align*}
\Vert\text{Op}(\sigma)f \Vert_{B^{r}_{p,q}(M)} &\lesssim \left( \sum_{m=0}^{\infty} 2^{mrq} \sup_{|\beta|\leq l,z\in M}\left\Vert \sum_{2^m\leq \langle \xi\rangle <2^{m+1}} d_{\xi} \text{Tr}[\xi(x)\mathscr{F}(\text{Op}(\partial^{\beta}_{z}\sigma(z,\cdot))f)(\xi)]   \right\Vert_{L^{p_2}}^q   \right)^{\frac{1}{q}}\\
&\leq \sup_{|\beta|\leq l,z\in M}\Vert\text{Op}(\partial^{\beta}_{z}\sigma(z,\cdot))f \Vert_{B^{r}_{p_2,q}}  \\
&\leq \left[\sup_{|\beta|\leq l,z\in M}  \Vert \text{Op}(\partial^{\beta}_{z}\sigma(z,\cdot)) \Vert_{B(B^{r}_{p_1,q},B^{r}_{p_2,q})}\right]\Vert f \Vert_{B^r_{p_1,q}}.
\end{align*}
So, we deduce the boundedness of $A=\text{Op}(\sigma)$. Now, we treat of a similar way the boundedness of $\text{Op}(\sigma)$ if $q=\infty.$ In fact, from the inequality
\begin{align*} &\Vert \sum_{2^m\leq \langle \xi\rangle <2^{m+1}} d_{\xi}\text{Tr}[\xi(x)\mathscr{F}(\text{Op}(\sigma)f)(\xi)] \Vert_{L^{p_2}}\\
&\lesssim \sup_{|\beta|\leq l,z\in M}\Vert \sum_{2^m\leq \langle \xi\rangle <2^{m+1}} d_{\xi} \text{Tr}[\xi(x)\mathscr{F}(\text{Op}(\partial^{\beta}_{z}\sigma(z,\cdot))f)(\xi)]   \Vert_{L^{p_2}}
\end{align*}
we have
\begin{align*} &2^{mr}\Vert \sum_{2^m\leq \langle \xi\rangle <2^{m+1}} d_{\xi}\text{Tr}[\xi(x)\mathscr{F}(\text{Op}(\sigma)f)(\xi)] \Vert_{L^{p_2}}\\
&\lesssim 2^{mr}\sup_{|\beta|\leq l,z\in M}\Vert \sum_{2^m\leq \langle \xi\rangle <2^{m+1}} d_{\xi} \text{Tr}[\xi(x)\mathscr{F}(\text{Op}(\partial^{\beta}_{z}\sigma(z,\cdot))f)(\xi)]   \Vert_{L^{p_2}}.
\end{align*}
So we get
$$ \Vert\text{Op}(\sigma)f \Vert_{B^{r}_{p,\infty}(M)} \lesssim \left[\sup_{|\beta|\leq l,z\in M}  \Vert \text{Op}(\partial^{\beta}_{z}\sigma(z,\cdot)) \Vert_{B(B^{r}_{p_1,q},B^{r}_{p_2,q})}\right]\Vert f \Vert_{B^r_{p_1,\infty}}. $$
With the last inequality we end the proof.
\end{proof}

\begin{theorem}
Let $G$ be a Compact Lie group, $n=\dim (G)$ and let $0\leq \rho, \delta \leq 1.$ Denote by $\varkappa$ the smallest even integer larger that $\frac{n}{2}.$ Let $1<p<\infty$ and $l=[\frac{n}{p}]+1. $ Let $A:C^{\infty}(G)$ into $C^{\infty}(G)$ be a pseudo-differential operator with global symbol $a(x,\xi)$ satisfying
\begin{equation}\label{ruzhanskycondition}
\Vert \mathbb{D}_{\xi}^{\alpha}\partial_{x}^{\beta}a(x,\xi) \Vert_{op}\leq C_{\alpha,\beta}\langle \xi\rangle^{-m-\rho|\alpha|+\delta|\beta|}, \,\,|\alpha|\leq \varkappa,|\beta|\leq l,
\end{equation}
with $m\geq \varkappa(1-\rho)|\frac{1}{p}-\frac{1}{2}|+\delta l.$ Then $A$ extends to a bounded operator from $B^{r}_{p,q}(G)$ into $B^{r}_{p,q}(G)$ for all $0< q\leq \infty$ and $r\in\mathbb{R}.$ Moreover, if we assume that
\begin{equation}\label{cardonacondition}
\Vert \mathbb{D}_{\xi}^{\alpha}\partial_{x}^{\beta}a(x,\xi) \Vert_{op}\leq C_{\alpha,\beta}\langle \xi\rangle^{-\rho|\alpha|}, \,\,|\alpha|\leq \varkappa,|\beta|\leq l,
\end{equation}  then $A$ extends to a bounded operator from $B^{r + \varkappa(1-\rho)|\frac{1}{p}-\frac{1}{2}|}_{p,q}(G)$  into $B^{r }_{p,q}(G)$  for all $1<p<\infty,$ $0<q\leq\infty$ and $r\in\mathbb{R}.$
\end{theorem}
\begin{proof}
If $A=\text{Op}(a)$ is a Fourier multiplier, i.e, $a(x,\xi)=a(\xi),$ by using Theorem \ref{t2} we have that $A$ is bounded operator from $L^{p_1}$ into $L^{p_2}$ and consequently $A$ extends to a bounded operator from $B^{r}_{p,q}(G)$ into $B^{r}_{p,q}(G)$ for all $0< q\leq \infty$ and $r\in\mathbb{R}.$ For the general case where $a(x,\xi)$ as in \eqref{ruzhanskycondition} depends on the spatial variable, we have, as in the previous proof that 
\begin{equation}\label{extension} \Vert\text{Op}(a)f \Vert_{B^{r}_{p,q}(G)} \lesssim \left[\sup_{|\beta|\leq l,z\in G}  \Vert \text{Op}(\partial^{\beta}_{z}a(z,\cdot)) \Vert_{B(B^r_{p,q},B^r_{p,q})}\right]\Vert f \Vert_{B^r_{p,q}}. 
\end{equation}
In fact, every multiplier $\text{Op}(\partial^{\beta}_{z}a(z,\cdot))$ is bounded on $B^{r}_{p,q}(G)$ because we only needs
\begin{equation}
\Vert\mathbb{D}^{\alpha}_{\xi}(\partial^{\beta}_{z}a(z,\xi))\Vert_{op}\lesssim\langle \xi\rangle^{-m-\rho|\alpha|},\,\,|\alpha|\leq \varkappa,\,|\beta|\leq l. 
\end{equation}
For the proof of this necessary condition, we use the fact that $m\geq \varkappa(1-\rho)|\frac{1}{p}-\frac{1}{2}|+\delta l.$ In fact,
\begin{align*}
\Vert\mathbb{D}^{\alpha}_{\xi}(\partial^{\beta}_{z}a(z,\xi))\Vert_{op}\lesssim \langle \xi\rangle^{-m-\rho|\alpha|+\delta|\beta|}\lesssim \langle \xi\rangle^{-m-\rho|\alpha|+\delta l }\lesssim \langle \xi\rangle^{-\varkappa(1-\rho)|\frac{1}{p}-\frac{1}{2}|-\rho|\alpha|}
\end{align*}
which shows the boundedness of the multiplier $(\partial^{\beta}_{z}a(z,\cdot))$ on $B^{r}_{p,q}(G).$ Since the family of operators $(\partial^{\beta}_{z}a(z,\cdot))_{z\in G}$ has norm uniformly bounded in $z$ we have,
$$ \sup_{|\beta|\leq l,z\in M}  \Vert \text{Op}(\partial^{\beta}_{z}a(z,\cdot)) \Vert_{B(B^r_{p,q},B^r_{p,q})}<\infty. $$
So, we end the proof for this case. If the symbol $a(\xi)$ satisfies \begin{equation}
\Vert \mathbb{D}_{\xi}^{\alpha}a(x,\xi) \Vert_{op}\leq C_{\alpha}\langle \xi\rangle^{-\rho|\alpha|}, \,\,|\alpha|\leq \varkappa,
\end{equation}  then the corresponding operator $T_a$ is bounded from $H^{m_p,p}(G)$ into $L^p(G),$  $m_p= \varkappa(1-\rho)|\frac{1}{p}-\frac{1}{2}|,$ (Corollary 5.1 of \cite{Ruz3}), so we have for $0<q<\infty$ the estimate
\begin{align*}
\Vert T_af \Vert^q_{B^{r}(G)}&=\sum_{l\geq 0} 2^{lrq} \Vert T_a[ (\text{Op}(\chi_{m})f)]\Vert^q_{L^{p}(G)}\\
&\lesssim \sum_{l\geq 0} 2^{lrq} \Vert  (\text{Op}(\chi_{m})f)\Vert^q_{H^{m_p,p}(g)}\\
&= \sum_{l\geq 0} 2^{lrq} \Vert \Lambda_{m_p}[ (\text{Op}(\chi_{m})f)]\Vert^q_{L^{p}(G)}= \sum_{l\geq 0} 2^{lrq} \Vert  (\text{Op}(\chi_{m})\Lambda_{m_p}f)\Vert^q_{L^{p}(M)}\\
&=\Vert \Lambda_{m_p} f \Vert^q_{B^{r}(G)}\lesssim \Vert f \Vert^q_{B^{r+m_p}(G)}.
\end{align*}
which proves the boundedness of $T_a.$ Now, we extend the boundedness result for non-invariant symbols $a(x,\xi)$ as in \eqref{cardonacondition} by using the inequality \eqref{extension}. The proof for $q=\infty$ is analogous.
\end{proof}

\begin{theorem}\label{T4}
Let $G$ be a compact Lie group, $n=\dim (G),$ $0\leq \rho< 1$ and $0\leq \nu<\frac{n}{2}(1-\rho).$  Let $A:C^{\infty}(G)$ into $C^{\infty}(G)$ be a pseudo-differential operator with global symbol $a(x,\xi)$ satisfying
\begin{equation}
\Vert \mathbb{D}_{\xi}^{\alpha}\partial_{x}^{\beta}a(x,\xi) \Vert_{op}\leq C_{\alpha,\beta}\langle \xi\rangle^{-\nu-\rho|\alpha|}, \,\,\alpha\in\mathbb{N}^{n},|\beta|\leq l,
\end{equation}
with $1<p<\infty$ and $l=[\frac{n}{p}]+1.$ Then $A$ extends to a bounded operator from $B^{r}_{p,q}(G)$ into $B^{r}_{p,q}(G)$ for all $|\frac{1}{p}-\frac{1}{2}|\leq \frac{\nu}{n}(1-\rho)^{-1}, $ $0< q\leq \infty$ and $r\in\mathbb{R}.$
\end{theorem}

\begin{proof}
Again, if $a(\cdot,\cdot)$ is independent of the spatial variable, the Fourier multiplier $A=\textrm{Op}(a)$ is bounded on $L^{p}(G)$ as consequence of Theorem \ref{t3}. Newly, by theorem \ref{T1} we obtain that the Fourier multiplier $A$ is bounded on $B^{r}_{p,q}(G).$ We know that for $l=[\frac{n}{p}]+1$ $$ \Vert\text{Op}(a)f \Vert_{B^{r}_{p,q}(G)} \lesssim \left[\sup_{|\beta|\leq l,z\in M}  \Vert \text{Op}(\partial^{\beta}_{z}a(z,\cdot)) \Vert_{B(B^r_{p,q},B^r_{p,q})}\right]\Vert f \Vert_{B^r_{p,\infty}} $$
provide that every multiplier $\text{Op}(\partial^{\beta}_{z}a(z,\cdot))$ is bounded on $B^{r}_{p,q}(G).$ But, this it follows from the fact that
\begin{equation}
\Vert \mathbb{D}_{\xi}^{\alpha}\partial_{x}^{\beta}a(x,\xi) \Vert_{op}\leq C_{\alpha,\beta}\langle \xi\rangle^{-\nu-\rho|\alpha|}, \,\,\alpha\in\mathbb{N}^{n},|\beta|\leq l.
\end{equation}
\end{proof}
\subsection{Examples}\label{examples} Now, we consider examples of differential problems which could not treated with the classical pseudo-differential calculus (based in the notion of local symbols). We follow \cite{Ruz3}.
\begin{example}
Let us consider $G=\textnormal{SU}(2),$ and let  $\{X,Y,Z\}$ be a basis of its Lie algebra $\mathfrak{g}=\textnormal{su}(2).$ Let us consider the differential operators
\begin{itemize}
\item $\mathcal{L}_{sub}=X^2+Y^2$ (sub-Laplacian) and
\item $H:=X^2+Y^2-Z,$
\end{itemize} which are  hypoelliptic operators by  H\"ormander's sum of squares theorem. A parametrix $P$ of $\mathcal{L}_{sub}$ has matrix valued symbol $\sigma\in S^{-1}_{\frac{1}{2},0}(\textnormal{SU}(2)).$ On every coordinate chart $U\subset \mathbb{R}^3$ of $\textnormal{SU}(2),$ $\mathcal{L}_{sub}$ has a local symbol in the class $S^{-1}_{\frac{1}{2},\frac{1}{2}}(U\times \mathbb{R}^3).$ The classical H\"ormander classes on compact manifolds $M$ require the condition $\rho\geq 1-\delta$ and $\rho>\delta$ (which implies that $\rho>\frac{1}{2}$), hence such calculus cannot be used for the analysis of the sub-Laplacian. The global description of the H\"ormander classes trough Ruzhansky-Turunen calculus gives together with  Theorem \ref{T3} that $P$ is a bounded operator on $B^r_{p,q}(\textnormal{SU}(2)),$ $r\in \mathbb{R},$ $0<q\leq \infty,$ and $1< p<\infty.$ Hence, if we consider the problem
\begin{equation}
\mathcal{L}_{sub}u=f,\,\,\, f\in B^r_{p,q}(\textnormal{SU}(2))
\end{equation}
and we assume that the problem has almost one solution $u\in B^s_{p,q}(\textnormal{SU}(2)),$ we have
\begin{equation}
\Vert u \Vert_{B^{r}_{p,q}(G)}\lesssim \Vert f  \Vert_{B^{r-1+\frac{\varkappa}{2}|\frac{1}{p}-\frac{1}{2}|}_{p,q}(G)}+ \Vert u \Vert_{B^{s}_{p,q}(G)}.
\end{equation}
On the other hand, since the operator $H:=X^2+Y^2-Z$  has parametrix with symbol in $S^{-1}_{\frac{1}{2},0}(\textnormal{SU}(2)),$ we have  the estimate
\begin{equation}
\Vert u \Vert_{B^{r}_{p,q}(G)}\lesssim \Vert Hu  \Vert_{B^{r-1+\frac{\varkappa}{2}|\frac{1}{p}-\frac{1}{2}|}_{p,q}(G)}+ \Vert u \Vert_{B^{s}_{p,q}(G)},
\end{equation}
as in the sub-Laplacian case for the following differential problem
\begin{equation}
Hu=f,\,\,\,f\in B^r_{p,q}(\textnormal{SU}(2)).
\end{equation}
\end{example}
We end this section with the following example on vector fields on arbitrary compact Lie groups.

\begin{example}
Let $X$ be a real left invariant  vector field on a compact Lie group $G.$ There exists an exceptional discrete set $\mathcal{C}\subset i\mathbb{R},$ such that $X+c$ is globally hypoelliptic for all $c\notin \mathcal{C}.$ We recall that an differential operator $A:\mathscr{D}'(G)\rightarrow \mathscr{D}'(G)$ is globally hypoelliptic, if $u\in \mathscr{D}'(G),$ $Au=f,$ and $f\in C^{\infty}(G)$ implies $u\in C^{\infty}(G)$. If $G=\textnormal{SU}(2),$ $X+c$ is globally hypoelliptic  if and only if $c\notin\mathcal{C}=\frac{1}{2}i\mathbb{Z}.$  Moreover, on a compact Lie group $G,$ the inverse $P=(X+c)^{-1}$ of $X+c$ has global symbol in $S^{0}_{0,0}(G).$ As in the sub-Laplacian case, the classical pseudo-differential calculus cannot be
used for the analysis of $X+c.$ However, if we use Theorem \ref{T4}, $P$ is a bounded operator from $B^{r+\varkappa|\frac{1}{p}-\frac{1}{2}| }_{2,q}(G) $  into $B^{r}_{2,q}(G)$ for all $r\in\mathbb{R}$ and $0<q\leq \infty.$ Hence, we obtain the (sub-elliptic) estimate
\begin{equation}
\Vert u \Vert_{B^{r }_{2,q}(G)}\lesssim \Vert (X+c)u \Vert_{B^{r+\varkappa|\frac{1}{p}-\frac{1}{2}| }_{2,q}(G)}.
\end{equation}

\end{example}

\noindent \textbf{Acknowledgments:} The author is indebted with Alexander Cardona for helpful comments on an earlier draft of this paper. This project was supported by Faculty of Sciences of Universidad de los Andes,\textit{ Proyecto: Una clase de operadores pseudo-diferenciales en espacios de Besov.} 2016-1, Periodo intersemestral.

\bibliographystyle{amsplain}

\end{document}